\documentclass[12 pt]{amsart}

\usepackage{amsmath,amsthm,amssymb,amsfonts}
\usepackage[all]{xy}

\newtheorem{thm}{Theorem}
\newtheorem*{thm maybe}{Perhaps a theorem}

\newtheorem{rem}{Remark}
\newtheorem{cor}{Corollary}
\newtheorem{prop}{Proposition}
\newtheorem{defn}{Definition}
\newtheorem{example}{Example}
\newtheorem{exa}[example]{Example}
\newtheorem{lemma}{Lemma}
\newtheorem{conjecture}{Conjecture}

\newcommand{\BZ}{\mathbb{Z}}
\newcommand{\BR}{\mathbb{R}}
\newcommand{\BN}{\mathbb{N}}
\newcommand\s{\vspace{1pc}}
\newcommand{\ignore}[1]{}
\newcommand{\strict}{\not\simeq}
\renewcommand{\simeq}{\longleftrightarrow}
\renewcommand{\implies}{\longrightarrow}
\newcommand{\len}{{\rm len}}
\newcommand{\ofrac}[2]{#1 / #2 }
\def\perb{{\perp {\rm B}}}

\def\Vn{V_\lambda^{(n)}}

\def\p{\mathbb{P}}
\def\e{\mathbb{E} \,}
\def\nullleftbrace{}
\def\nullrbrace{}
\def\nc{}

\title[Bernoulli Matrices and Novel Integer Partitions]{On The Singularity of Random Bernoulli Matrices --- Novel Integer Partitions and  Lower Bound Expansions}
\author{Richard Arratia}
\email{rarratia@math.usc.edu}
\author{Stephen DeSalvo}
\email{stephendesalvo@gmail.com}
\address{University of Southern California, Department of Mathematics}
\date{April 18, 2012}


\begin{document}

\begin{abstract}
We prove a lower bound expansion on the probability that a random $\pm 1$ matrix is singular, and 
conjecture that such expansions govern the actual probability of singularity.
These expansions are based on naming the most likely, second most likely, and so on,
ways that a Bernoulli matrix can be singular;  the most likely way is to have a null vector of the form $e_i \pm e_j$, which corresponds to the integer partition 11, with two parts of size 1.  The second most likely way is to have a null vector of 
the form $e_i \pm e_j \pm e_k \pm e_\ell$, which corresponds to the partition 1111. The fifth most likely way corresponds to the partition 21111.  

We define and characterize the ``novel partitions'' which show up in this series.   
As a family, novel partitions suffice to detect singularity, i.e., any singular Bernoulli matrix has a left null vector whose underlying
integer partition is novel.  And, with respect to this property, the family of novel partitions is minimal.

We prove that the only novel partitions 
with  six or fewer parts are 11, 1111, 21111, 111111, 221111,  311111, and 322111. We prove that there are fourteen novel partitions having seven parts.   

We formulate a conjecture about which partitions are ``first place and runners up," 
in relation to the  Erd\H{o}s-Littlewood-Offord bound. 

We prove some bounds on the interaction between left and right null vectors.

\end{abstract}

\maketitle

\tableofcontents

\section{Introduction}
To introduce our problem, we quote 
verbatim\footnote{Apart from correcting a typographical error, and using our own display equation numbering and reference numbering.}
the opening 3 paragraphs of a paper by Kahn, Koml{\'o}s, and Szemer{\'e}di \cite{Komlos999}:

\s
{\small
``1.1. {\bf The problem.} For $M_n$ a random $n\times n$ $\pm 1$-matrix (``random" meaning with respect to uniform distribution), set $$P_n = Pr(M_n\ {\rm is\  singular}).$$  The question considered in this paper is an old and rather notorious one: What is the asymptotic behavior of $P_n$?

It seems often to have been conjectured that \begin{equation}\label{Komlos1} P_n = (1+o(1)) n^2/2^{n-1},\end{equation} that is, that $P_n$ is essentially the probability that $M_n$ contains two rows or two columns which are equal up to a sign.  This conjecture is perhaps best regarded as folklore.  It is more or less stated in \cite{KomlosManuscript} and is mentioned explicitly, as a standing conjecture, in \cite{Odlyzko}, but has surely been recognized as the probable truth for considerably longer.  (It has also been conjectured (\cite{matzeroone}) that $P_n / (n^2 2^{-n}) \to \infty$.)

Of course the guess in \eqref{Komlos1} may be sharpened, e.g., to
\begin{equation}
\label{Komlos2}
P_n - 2^2 \binom{n}{2} \left( \frac{1}{2} \right)^n \sim 2^4\binom{n}{4} \left(\frac{3}{8}\right)^n,
\end{equation} 
the right-hand side being essentially the probability of having a minimal row or column dependency of length 4."
}
\s

The above quoted paper was  the first to show that $P_n$ decays exponentially,
with an upper bound of $.999^n$.  This was later improved by Tao and Vu \cite{Tao958} to $(.958+o(1))^n$ and again \cite{Tao3Quarters} to $(3/4 + o(1))^n$. (See also \cite{TaoBook}).  Recently Bourgain, Vu, and Wood \cite{Wood} provided a further improvement to $\left(\frac{1}{\sqrt{2}} +o(1)\right)^n$, which is currently the most accurate bound.

Instead of focusing on upper bounds, we consider \emph{lower} bounds. 
Our paraphrase of the opening of \cite{Komlos999}:  \eqref{Komlos1} says,
for a Bernoulli matrix to be singular, the most likely way is to have a left or right null vector of the form $e_i \pm e_j$ for some $1\leq i, j\leq n$, $i\neq j$, which we say is ``of the \emph{template} 11,'' and
\eqref{Komlos2} says that the second most likely way to be singular is to have a left or right null vector of the form $e_i \pm e_j \pm e_k \pm e_\ell$ for distinct indices $i,j,k,\ell$, with $1\leq i,j,k,\ell\leq n$, i.e., of the   template 1111.  

We use the standard notation for integer partitions:  writing $\lambda=(\lambda_1,\lambda_2,\ldots,\lambda_k)$  implies that the integers $\lambda_i$ satisfy
$\lambda_1 \ge \lambda_2 \ge \cdots \ge \lambda_k \ge 1$.
When there is no confusion, as in $\lambda=(1,1,1,1)$, we will drop the parentheses and commas and simply write 1111 for the partition. 
We say that a vector is of the template $\lambda$ if it is a non-zero multiple of a vector of the form  $\lambda_1 e_{i_1} \pm \lambda_2 e_{i_2} \pm \ldots \pm \lambda_k e_{i_k}$ for some set of distinct indices $1 \leq i_1,\ldots,i_k \leq n$.  We define $R_\lambda$ (resp. $L_\lambda$)  to be the event that the random matrix has a right (resp. left) null vector of template $\lambda$, with 
\begin{equation}\label{def D}
D_\lambda := R_\lambda \cup L_\lambda
\end{equation}
being the event that the random matrix has one or more right or left null vectors of template $\lambda$. 

The expansion \eqref{Komlos2}
has the form $Q_1(n) \,(1/2)^n + Q_2(n) \, (3/8)^n$, where the $Q_i$ are polynomials in $n$.  When one continues the expansion to higher exponential order, two features emerge.  First, the templates, corresponding to 11, 1111, $\dots$, have a rich structure: the real pattern is not simply an even number of 1s, and this first appears in the fifth term, coming from the template 21111.  The second feature,  also first appearing with the fifth term,  is the need to distinguish  between
 the expected number of occurrences of a right or left null vector of template $\lambda$, which for $\lambda=11$ is $2^2 \binom{n}{2}(1/2)^n$,
and $\p_n(D_\lambda)$, 
the probability of one or more such occurrences; see Equations \eqref{incexc} and \eqref{incexcE}.
This is because the exponential decay rate for 21111, which is $(1/4)^n$, is small enough to force consideration 
of the difference between the \emph{expected number} of occurrences of a right or left null vector of template 11, and the \emph{probability} of one or more such occurrences.

The natural extensions of \eqref{Komlos2} are our Conjectures \ref{c1} and \ref{c2}, immediately below.
\begin{conjecture}
\label{c1}
Let $S$ denote the event that the $n$ by $n$ random Bernoulli matrix $M = M_n$ is singular, with $P_n = \p(S) = \p_n(S)$.  
Then for every $\epsilon>0$,

\begin{eqnarray}
\p(S\setminus D_{11}) &=& o\left( \left(\frac{3+\epsilon}{8}\right)^n\right) \label{conj 1 line 1}\\ 
\p(S\setminus (D_{11}\cup D_{1111})) &=& o\left( \left(\frac{5+\epsilon}{16}\right)^n\right) \nonumber \\
\p(S\setminus (D_{11}\cup D_{1111}\cup D_{1^6})) &=& o\left( \left(\frac{35+\epsilon}{128}\right)^n\right) \nonumber \\
\p(S\setminus (D_{11}\cup D_{1111}\cup D_{1^6}\cup D_{1^8}) &=& o\left( \left(\frac{1+\epsilon}{4}\right)^n\right) \nonumber 
\end{eqnarray}
\begin{eqnarray*}
\p(S\setminus (D_{11}\cup D_{1111}\cup D_{1^6} \cup D_{1^8} \cup D_{21111})) &=&o\left( \left(\frac{63+\epsilon}{256}\right)^n\right)
\end{eqnarray*}
\begin{eqnarray*}
\p(S\setminus E_6) &=&o\left( \left(\frac{15+\epsilon}{64}\right)^n\right) \\
\p(S\setminus E_7) &=&o\left( \left(\frac{231+\epsilon}{1024}\right)^n\right)\\
\p(S\setminus E_8) &=&o\left( \left(\frac{7+\epsilon}{32}\right)^n\right)
\end{eqnarray*}
and so on, where $E_6 = D_{11}\cup D_{1111}\cup D_{1^6} \cup D_{1^8}  \cup D_{21111} \cup D_{1^{10}}$, $E_7 = E_6 \cup D_{2\, 1^6}$, and $E_8 = E_7 \cup D_{1^{12}}.$
\end{conjecture}

In Section \ref{sect novel} we define what we call \emph{novel} integer partitions.  We prove that the set of these is, in a sense, necessary and sufficient for detecting singularities.  The precise statements are 
Theorem \ref{thm suffices} (sufficiency), and Theorem \ref{thm minimal}  (a minimality property which loosely can be called necessity).   The denumerability of the set of novel partitions, together with
the Erd\H{o}s, Littlewood, Offord bound (see Proposition \ref{LEOthm}), allows us to extend Conjecture 1.

\begin{conjecture}\label{c2}
For any enumeration $\lambda(1),\lambda(2),\ldots$ of the set of novel partitions, 
for every $r>0$, there exists $K>0$ such that
$$\p\left(S\setminus \bigcup_{i=1}^K D_{\lambda(i)}\right) = o(r^n).$$
\end{conjecture}

Of course, nice enumerations are those for which $K = K(r)$ is minimal, and this corresponds to listing the partitions in nonincreasing order of exponential rate \eqref{def r}.   Lemma 
\ref{l1} in the next section \emph{proves} that  the first 8 terms on a nice list are $\lambda(1)=11, \lambda(2)=1111,\ldots, \lambda(5)=21111,\ldots, $ $\lambda(8)=1^{12}$.   Table \ref{table r} gives a \emph{plausible} listing, in order, out to the 59$^{th}$
novel partition, and in this table, the first appearance of a 3 is in $\lambda(25)$.

 Section \ref{sect expand} presents an explicit \emph{lower} bound expansion of $P_n$, whose exponential decay rates are based on the novel integer partitions of Section \ref{sect novel}.  In Section \ref{sect ie} we derive the polynomial coefficients of our lower bound expansion.  In Section \ref{sect left right} we give some bounds on the interaction of
potential left and right null vectors, hoping to supply a tool for use in bounding $\p(S \setminus D_{11})$.

\section{Lower bound expansions}\label{sect expand}
 The expansion in \eqref{Komlos2} can be continued by considering events $D_{1111}$, that $M$ has a left or right null vector of the form $e_i\pm e_j \pm e_k \pm e_\ell$, with $D_{1^6}$, $D_{1^8}$, $D_{2\,1^4}$, $D_{1^{10}}$, $D_{2\,1^6}$, and $D_{1^{12}}$ defined similarly.  Letting $E_8 = D_{11} \cup D_{1^4} \cup D_{1^6} \cup D_{1^8} \cup D_{2\,1^4} \cup D_{1^{10}} \cup D_{2\,1^6} \cup D_{1^{12}}$, our expansion can be stated as

\begin{thm}
\label{thm1}
For each $n$, 
$$
P_n \ge  \p(D_{11}) \ge   4\binom{n}{2} \left(\frac{1}{2}\right)^n - \left(12\binom{n}{2}^2 - 4 
\binom{n}{2}\right) \left(\frac{1}{4}\right)^n.
$$
For each $n$, the event $E_8$ is a subset of the event that $M$ is singular, hence trivially, 
$$
P_n \ge \p_n(E_8).
$$
Lower \emph{and} upper bounds on $\p_n(E_8)$ are given by the statement: for all 
$\epsilon >0$, 
\begin{eqnarray*}
\p_n(E_8) &=& Q_1(n)\left(\frac{1}{2}\right)^n + Q_{2}(n)\left(\frac{3}{8}\right)^n + Q_3(n)\left(\frac{5}{16}\right)^n + Q_4(n) \left(\frac{35}{128}\right)^n  \\
&& +Q_5(n) \left(\frac{1}{4}\right)^n+Q_6(n)\left(\frac{63}{256}\right)^n + Q_7(n)\left(\frac{15}{64}\right)^n+Q_8(n)\left(\frac{231}{1024}\right)^n \\
&&+o\left(\left(\frac{7+\epsilon}{32}\right)^n\right),
\end{eqnarray*}
where the polynomial coefficients of the exponentially decaying factors are given by
$$\large \begin{array}{llll}
Q_1(n)  =  2^2\binom{n}{2}, & Q_2(n) = 2^4 \binom{n}{4} \end{array}$$
$$\large \begin{array}{llll}
Q_3(n)  =  2^6 \binom{n}{6}, & Q_4(n)  =  2^8 \binom{n}{8},\\
\end{array}$$
$$\large  Q_5(n)  =  2^5 \binom{5}{1}\binom{n}{5} 
 -4\left(2\binom{n}{2}^2 + 8\binom{n}{4} + 5\binom{n}{3}\right),$$
$$\large \begin{array}{llll}
Q_6(n)  =  2^{10} \binom{n}{10},   & Q_7(n)  =  2^7 \binom{7}{1}\binom{n}{7}, & Q_8(n)  =  2^{12} \binom{n}{12}. & 
\end{array}$$
\end{thm}
\begin{proof}
This result follows easily from the combination of Lemmas \ref{l1} -- \ref{lemmatrivial}, given in Sections
\ref{sect novel} and \ref{sect ie}.
 \end{proof}

\section{Templates, Bernoulli orthogonal complements, and novel partitions}\label{sect novel}


When the expansions of Conjecture \ref{c1} and Theorem \ref{thm1} are carried out to high order, an obvious necessary condition for a partition
$\lambda=(\lambda_1,\lambda_2,\ldots,\lambda_k)$ to appear is that it be \emph{fairly divisible}, in the sense that for some combination of signs, $0= \lambda_1 \pm \lambda_2 \pm \cdots \pm \lambda_k$.  However, this is not sufficient;  some fairly divisible
partitions, such as $211$ and $321$, will never appear.  We call the  partitions that eventually appear \emph{novel}. The definitions below will let us characterize these novel partitions, and, to a limited extent, compute them explicitly.

\begin{defn} 
\label{def template}{\bf Integer partition, as a \emph{template} for vectors.}\\
For a given partition $\lambda$ with $k$ parts, let $V_\lambda \subset \BN \times \BZ^{k-1}$ denote the set of all vectors formed by reordering the parts of lambda, together  with all combinations of plus and minus with the requirement that the first coordinate  always has a plus.\footnote{We write $\BN:= \{1,2,\ldots\}$ for the set of strictly positive integers.}
\end{defn}

If $\lambda$ has $c(i)$ parts of size $i$, so that $\len(\lambda) := c(1)+c(2)+\cdots=k$, then
$$
| \ V_\lambda \ | =   2^{k-1} \   \frac{k!}{c(1)! c(2)! \cdots} .
$$

\noindent {\bf Notation:  coordinate injection, from $\BR^k$ to $\BR^n$.}\\
We often want to pad our vectors of length $k$ with zeros, to get a vector of length $n$.  We say that a $k$ by $n$ matrix
$C$, with all entries 0 or 1, is a \emph{coordinate injection matrix}, if every row has exactly one 1, and no column has more than one 1, and $C_{ij}=C_{i'j'}=1$ with $i<i'$ implies $j < j'$.  (This last requirement is imposed, since our $V_\lambda$ already accounts for all rearrangements of the parts.)   There are ${n \choose k}$ such matrices.   We speak of vectors of length $n$, of the form $vC$
for some $v \in V_\lambda$ as \emph{having template} $\lambda$.  
\begin{defn}\label{def n template}{\bf Templates, used in $n$ dimensions.}\\
We write $\Vn$ for the subset  of $\BZ^{n}$ of length $n$ vectors with \emph{template} $\lambda$.
Note, vectors in $\Vn$ may have first coordinate zero, but the first non-zero coordinate must be strictly positive.
\end{defn}

The number of vectors of length $n$, having template
$\lambda$, is
\begin{equation}\label{number lambda n}
 | \Vn| =  {n \choose k} \ | \ V_\lambda \ | = 2^{k-1} \ \frac{(n)_k}{c(1)! c(2)! \cdots},
\end{equation}
where we write $(n)_k$ for $n$ falling $k$.

For an integer partition $\lambda$ with parts $\lambda_1, \lambda_2, \ldots \lambda_k$, and $X = (\epsilon_1,\ldots,\epsilon_k)$ a vector of independent Bernoulli random variables, let $\lambda\cdot X = \lambda_1\epsilon_1+\ldots + \lambda_k\epsilon_k$ denote the weighted sum, and define  
\begin{equation}\label{def r}
r_\lambda := \p(\lambda \cdot X = 0).
\end{equation}
We can then compute, for example, $r_{11}=1/2$, $r_{1^{2m}}={2m \choose m}/2^{2m}$, $r_{21111}=1/4$.

\begin{defn}
\label{def2} {\bf Bernoulli orthogonal complement.}\\
For a vector $v \in \BZ^k$, 
$$v^\perb = \{x\in \{-1,1\}^k: v\cdot x = 0\}.$$
This definition can also be applied when $v=\lambda=(\lambda_1,\ldots,\lambda_k)$ with $\lambda_1 \ge \cdots \ge \lambda_k > 0$ is  an integer partition with $k$ parts, in which case
the probability $r_\lambda$ defined by \eqref{def r}
is given by 
\begin{equation}\label{size perp}
r_\lambda = \frac{| v^\perb|}{2^k}.
\end{equation}

\end{defn}

\begin{rem}
Clearly $x \in v^\perb$ iff $-x \in v^\perb$, that is  $-v^\perb=v^\perb$.
For a partition $\lambda$, all $v$ in $V_\lambda$ have the same size $|v^\perb|$ for their Bernoulli orthogonal complement.  Indeed, the various sets $v^\perb$ for $v \in V_\lambda$ are related, by permuting the $k$ coordinates, and applying, for some fixed $I \subset \{2,3,\ldots,k\}$, sign flips to all the coordinates indexed by $I$.  Hence, if $\lambda^\perb \ne \emptyset$, then
$\{-1,1\}^k = \cup_{v \in V_\lambda} v^\perb$.
\end{rem}

\begin{defn}\label{def A}{\bf The matrix $A^{(\lambda)}$ for  
$\lambda^\perb$.}\\
For an integer partition $\lambda$ of length $k$, with $2p = |\lambda^\perb| >0$, the matrix $A^{(\lambda)}$
for the Bernoulli orthogonal complement of $\lambda$ is the $k$ by $p$ matrix whose columns are  those elements of
$\lambda^\perb$ whose first coordinate is $+1$, taken in lexicographic order, with $+1$ preceding $-1$.
\end{defn}

 \begin{exa}{\bf Displaying the Bernoulli orthogonal complement.}\label{ex display}\\
 When $\lambda = 1111$, we have
 $$1111^\perb = \begin{array}{cccccc} 
 \{ & (+1, & +1, & -1, & -1), & \\ & (+1, & -1, & +1, & -1), & \\ & (+1, & -1, & -1, & +1), & \\
  & (-1, & -1, & +1, & +1), & \\  & (-1, & +1, & -1, & +1), & \\  & (-1, & +1, & +1, & -1) &  \}\end{array} = \begin{array}{cccccc} \{ & + & + & - & -, & \\ & + & - & + & -, & \\ & + & - & - & +, & \\
  & - & - & + & +, & \\  & - & + & - & +, & \\  & - & + & + & - &  \}\end{array},$$
where the second representation omits the parentheses and commas for each $k$-tuple, and also shows only the signs. 

Say that $\lambda$ has length $k$, and $2p=|\lambda^\perb|$.  Showing only those elements of $\lambda^\perb$ that begin with $+$, and transposing, we have a $k$ by $p$ display, to be thought of as an economical representation of the set $\lambda^\perb$;  we use this display in Example \ref{example strict}.  Treating the same $k$ by $p$ array as a matrix, we have 
$A^{(\lambda)}$, as defined in Definition \ref{def A}.  For instance, 
 $$
   A^{(1111)}=
 \left(
\begin{array}{ccc}
 + & + & + \\
 + & - & - \\
 - & + & - \\
 - & - & +
\end{array}
\right).
$$
\end{exa}

\begin{defn}\label{def equiv}{\bf Equivalence of templates.}\\
For partitions $ \lambda,\mu$ with the same number of parts, we say $\lambda \simeq 
\mu$ iff $\exists v \in V_\lambda, w \in V_\mu $, such that $v^\perb=w^\perb$. Clearly, this $\simeq$ is an equivalence relation on integer partitions.  (Note,  $\lambda \simeq 
\mu$ iff $\exists w \in V_\mu $  such that $\lambda^\perb=w^\perb$, that is, we need only apply rearrangement and sign flips to one  of $\lambda,\mu$.)
\end{defn}

\begin{example}{\bf Equivalence is more than  just multiples.}\\
Trivially, scalar multiples of any partition are all equivalent to each other.   But equivalence involves more. 
Let $\lambda=321,\mu = 211$.  Then $321 \simeq 211$ since
 $$\mu^\perb =  \lambda^{\perb}=
\begin{array}{ccccc}\{ & + & - & -, \\ & - & + & + & \}\end{array},
$$
with no need to apply rearrangements or sign flips.
Rearrangement and sign flips may change the Bernoulli complement.  For instance, 
\begin{eqnarray*}
V_\mu& =& \{(2,1,1),(1,2,1),(1,1,2),
(2,1,-1), (1, 2, -1), ( 1, 1, -2),\\
& & (2,-1,1),  (1,-2,1), (1,-1,2),
(2,-1,-1), (1,-2,-1), (1,-1,-2)\}.
\end{eqnarray*}
and with $v=(1,-2,-1)\in V_\mu$, we have
$$v^\perb = \begin{array}{ccccc} \{ & + & + & -, & \\ & - & - & + & \} \end{array} \ne \mu^\perb.$$
\end{example}

\begin{example}{\bf Rearrangements are needed in the definition of equivalence.}\footnote{
This example was found by considering partitions of the form $(a+x_1,b+x_1,b,b,a-x_1,b-x_1)$ and $(a+x_2,a-x_2,b+x_2,b,b,b-x_2)$, where $a\geq b,x_1,x_2$ are chosen so that the two b's must cancel, but are in a different monotonic order in each partition.  Here we have taken $a=6, b=4, x_1=3, x_2=1$.}
\\
The partitions
$$\mu = 9\ 7\ 4\ 4\ 3\ 1$$ $$\lambda = 7\ 5\ 5\ 4\ 4\ 3$$
are such that $\mu^\perb \neq \lambda^\perb$, but for $v = (9,7,3,4,4,1) \in V_\mu$, $v^\perb = \lambda^\perb$, hence $\mu \simeq \lambda$.
\end{example}

 \begin{defn}\label{def reduce}{\bf Reduction of templates.}\\
For any partitions  $\mu, \lambda$ with $\mu$ having $m$ parts and $\lambda$ having $k$ parts, $m\geq k >0$, we say that $\mu \implies \lambda$ (read $\mu$ \emph{reduces to} $\lambda$ or $\mu$ \emph{implies} $\lambda$) iff either
\begin{equation}\label{reduce 1}
  k=m \mbox{ and }  \exists  v\in V_\lambda, \mu^\perb \subset v^\perb,
\end{equation} 
or else
\begin{equation}\label{reduce 2}
\exists  I\subset \{1,\ldots,m\},\  v\in V_\lambda, \ \ \ {\rm Proj}_I \mu^\perb \subset v^\perb.
\end{equation} 
\end{defn}
Clearly, the relation $\implies$ is transitive.
Our use of the subset symbol $\subset$ includes equality.
We note that ($\lambda \implies \mu$ and $\mu \implies \lambda$) iff  $\lambda \simeq \mu$, so that definitions 
\ref{def equiv} and \ref{def reduce} are compatible.  

\begin{rem}\label{remark implies}
Definition \ref{def reduce} is set up so that it is obvious that if $\mu \implies \lambda$, and $w \in V_\mu^{(n)}$, 
and $M$ is an
$n$ by $n$ Bernoulli matrix with $wM=0$, then there exists $v \in \Vn$ with $vM=0$.
\end{rem}

\begin{defn}\label{def strict reduction}{\bf Strict reduction.}\\
We define a relation of
\emph{strict} reduction, $\mu \strict \lambda$ (read
$\mu$ \emph{strictly reduces} to $\lambda$) iff $\mu \implies \lambda$ and \emph{not} $\lambda \implies \mu$.
Hence, $\mu \strict \lambda$ iff \eqref{reduce 1} with proper subset containment of the Bernoulli complements, or
\eqref{reduce 2} holds.  Clearly, the relation $\strict$ is transitive and irreflexive.
\end{defn}

\begin{example}{\bf Strict reduction using \eqref{reduce 1}.\\   $\mu := 332211 \strict \lambda := 221111$.} \\
\label{example strict}

$$ A^{(332211)} = 
\left(
\begin{array}{ccccc}
 + & + & + & + & + \\
 + & - & - & - & - \\
 - & + & + & - & - \\
 - & - & - & + & + \\
 - & + & - & + & - \\
 - & - & + & - & +
\end{array}
\right) ,
A^{(221111)}=
\left(
\begin{array}{ccccccc}
 + & + & + & + & + & + & + \\
 + & - & - & - & - & - & - \\
 - & + & + & + & - & - & - \\
 - & + & - & - & + & + & - \\
 - & - & + & - & + & - & + \\
 - & - & - & + & - & + & +
\end{array}
\right) 
$$

Upon visual inspection, it is easily seen that each column of $  A^{(332211)}$ appears as a column in $ A^{(221111)}$,  which shows that $332211^\perb \subset 221111^\perb$.

\end{example}

\begin{example}
{\bf Strict reduction using \eqref{reduce 2}.}\\
The partition 211 reduces to the partition 11. \\
$$ 
A^{(211)}=
\left(
\begin{array}{c}
 + \\
 - \\
 -
\end{array}
\right), \ \ A^{(11)} = \left(
\begin{array}{c}
 + \\
 -
\end{array}
\right).$$
Take $I=\{1,2\}$, so that projection onto I ``forgets" the third coordinate in 
$211^{\perb}$.  
We then have ${\rm Proj}_I 211^{\perb} = 11^{\perb}$, and $211\strict 11$.
\end{example}

\begin{example}\label{example coin}{\bf The consequence of not implying 11.} \\
If $\lambda$ does not imply 11, then every two rows of   $A^{(\lambda)}$ are linearly independent.  Thus, for every $i\neq j$, both $\lambda_i + \lambda_j$ and $\lambda_i - \lambda_j$ are expressible as  a plus-minus combination of the remaining parts.  (The proof of Proposition \ref{prop no 5}  uses this.)
\end{example}

There is a natural description of this principle in terms of coin weighing problems (see for example \cite{coinsurvey}).  You have $k$ coins of various positive integer weights.  Not implying 11 means that if an adversary selects any two coins and places them on the same or opposite sides of a balance scale, you can place  \emph{all} of the remaining coins on the scale so that it balances.

We now come to the definition that effectively governs explicit expansions such as those in Conjecture
\ref{c1} and Theorem \ref{thm1}.

\begin{defn}\label{def novel}{\bf Novel partitions.}\\
We call an integer partition $\lambda$ a \emph{novel} partition if and only if
there does not exist any other partition $\lambda'$ with $\lambda \strict \lambda'$,
and among all partitions equivalent to $\lambda$, in the sense of Definition \ref{def equiv}, $\lambda$  is  lexicographically first.
\end{defn}

\begin{thm}[Sufficiency of the set of novel partitions]\label{thm suffices}
The set of all novel partitions is sufficient, acting as possible left null vectors, to detect singularity for 
Bernoulli matrices $M$.  That is, if such a matrix is singular, say of size $n$ by $n$, then there exists a novel
partition $\lambda$ with $\len(\lambda)\leq n$, and $v \in \Vn$ with $vM=0$.
\end{thm}

\begin{proof}
If $M$ is singular, then there is a nonzero vector $w \in \BZ^n$ with $wM=0$.  Taking absolute values of the coordinates,
deleting zeros if they occur, and listing in nonincreasing order yields an integer partition $\lambda$, and  $w\in \Vn$.  If $\lambda$ is novel, we are done.  If $\lambda$ is not novel, then it must reduce to a novel partition $\mu$, and then
Remark \ref{remark implies} applies.
\end{proof}
 
 
\begin{thm}{\bf Intrinsic characterization of novel partitions.}\\
\label{rank equiv}
An integer partition $\lambda$ with $k$ parts is novel iff  the matrix $A^{(\lambda)}$, specified in Definition \ref{def A}, has rank $k-1$, and $\gcd(\lambda) = 1$.
\end{thm}

\begin{proof}
Let $A\equiv A^{(\lambda)}$, with rows $r_1,\ldots,r_k$.
To prove the only if direction, suppose rank $A < k-1$. Then there exists $j<k$, and integers $c_1,\ldots,c_j\neq 0$, $\pi_1,\ldots, \pi_j$ distinct elements of $\{1,\ldots,k\}$, such that $c_1 r_{\pi_1} + \ldots c_j r_{\pi_j} = 0.$  Letting $v= (c_1,\ldots,c_j)$, we have $v\in V_\mu$, $\len(\mu) = j<k$, and $\lambda \strict \mu$, so that $\lambda$ is not novel.

If $\gcd(\lambda)>1$, then $\mu = \frac{1}{\gcd(\lambda)}\lambda$, $\mu\simeq\lambda$, $\mu$ is earlier in lexicographic order, so $\lambda$ is not novel.

In the other direction, suppose rank $A=k-1$, $\gcd(\lambda)=1$, but assume $\lambda$ is not novel.  Then either
\begin{enumerate}
\item $\lambda \strict \mu$, $\len(\mu) < \len(\lambda)$, but then there exists $v\in V_\mu^{(k)}$, with $k-1$ or fewer nonzero components such that $vA=0$, which implies rank $A<k-1$;
\item $\lambda \strict \mu$, $\len(\mu) = \len(\lambda)$, $\lambda^\perb \subsetneq \mu^\perb$.  Then $A^{(\mu)}$ and $A^{(\lambda)}$ both have rank $k-1$, with $\mu A^{(\mu)} = 0$ and $\lambda A^{(\lambda)} = 0$.  By the inclusion, we also have $\mu A^{(\lambda)} = 0$.  But then if we consider the vector $v= \mu_1 \lambda - \lambda_1 \mu$ of length $k$ with first coordinate 0, $v$ has at most $k-1$ nonzero entries, and $vA^{(\lambda)} = 0$.  Since we assumed $A^{(\lambda)}$ has rank $k-1$, we conclude $v=0$.

\end{enumerate}

\end{proof}

\begin{cor}
An integer partition $\lambda$ is either 
\begin{enumerate}
\item novel (or a multiple of a novel),
\item implies a novel partition $\mu$ of \emph{strictly} smaller length, or
\item is not fairly divisible, i.e., $\lambda^\perb = \emptyset$.
\end{enumerate}
\end{cor}

The only part that is not trivial is (2).  We already showed that partitions like 332211 can strictly reduce to a  partition of the \emph{same} length, but without Theorem \ref{rank equiv} it is not a priori obvious that there will always be a strict reduction in the length of the partition.
 
\begin{thm}[Minimality of the set of novel partitions]
\label{thm minimal}
The family of novel partitions is \emph{minimal}, in the sense that if any single one of the $\lambda$s in that set is removed, then the family, acting as possible left null vectors, does not detect all singularities.
\end{thm}

\begin{proof}
Fix a novel partition $\lambda$, of length $k$.    We will construct a singular Bernoulli matrix $M$ with the property that every left null vector of $M$, having integer coordinates with greatest common divisor 1, is of template $\lambda$.
 
Let $2p=|\lambda^\perb|$; we can write $\lambda^\perb=\{x_1,x_2,\ldots,x_{2p}\}$, where $x_i \in \{-1,1\}^k$, $x_i \neq x_j$, $i\ne j$.  As in Example \ref{ex display}, let $A \equiv A^{(\lambda)}$ denote the $k$ by $p$ matrix of rank $k-1$ with columns given by $x_1,x_2,\ldots,x_{p}$, the vectors with first entry positive. 
\begin{enumerate}
\item If $p \le k$, then add $k-p$ columns which are duplicates of column $p$, and call this square matrix $M$.  
Note, since $A$ had rank $k-1$, either $A$ was $k$ by $k$ and $M=A$, or else we added exactly one column.
In either subcase, $M$ is $k$ by $k$ of rank $k-1$.

\item If $p>k$, since the rank of $A$ is $k-1$, the first $k-1$ rows, denoted $r_1,r_2,\ldots,r_{k-1}$, form an independent set in $\BR^p$.  The existence of an independent set of $p$ vectors in $\BR^p$, whose entries consist of plus and minus 1,  is guaranteed by the existence of nonsingular Bernoulli matrices of all sizes; denote such a set as  $\{s_1,\ldots,s_p\}$.  By the basis extension theorem, re-indexing the $s_i$ as needed, there is an independent set of the form $B =\{r_1,r_2,\ldots,r_{k-1},s_k,s_{k+1},\ldots,s_p\}$. Replacing $s_k$ with $r_k$, the rows
$r_1,r_2,\ldots,r_{k-1},r_k,s_{k+1},\ldots,s_p$ form a $p$ by $p$ Bernoulli matrix $M$ with rank $p-1$.  
\end{enumerate}

In either case, we have a square matrix $M$ of corank 1.  Suppose we have a left null vector $w$ with integer
coordinates. In case (1), this implies $\lambda^\perb \subset w^\perb$, and since $\lambda$ was novel, this implies $w \in V_\lambda$.  In case (2), write $w=(w_1,\ldots,w_k,\ldots,w_p)$.  The condition $wM=0$ says that, in row space, $0 = w_1 r_1 + \cdots + w_k r_k + w_{k+1} s_{k+1} +\cdots + w_p s_p$, and the 
independence of the set $B$ now implies that $0 = w_1 r_1 + \cdots + w_k r_k$ and $w_{k+1}=\cdots=w_p=0$.
With $v = (w_1,\ldots,w_k)$, we have $\lambda^\perb \subset v^\perb$, and since $\lambda$ was novel, this implies $v \in V_\lambda$ and $w \in V_\lambda^{(p)}$. 
\end{proof}

\begin{rem}
In Theorem \ref{thm minimal}, we specified  testing for null vectors on one particular side, since in the case $n=4$ any instance of singularity detected by a null vector of template 1111 on one side implies that there is a null vector
of template 11 on the other side. Without having specified a side, one could say that 1111 is not  necessary to detect singularity of 4 by 4 Bernoulli matrices; the template 11 by itself suffices.
\end{rem}

We have just defined and characterized novel partitions, which form the foundation for the expansion in Theorem \ref{thm1}.  The next set of theorems bounds the exponential decay from each term.

\begin{prop}[Erd\H{o}s, Littlewood, Offord \cite{LEO}]
\label{LEOthm}
Let $x_1,x_2,\ldots$ be real numbers, $|x_i|\geq 1$, and $\epsilon_1,\epsilon_2,\ldots$ be $+1$ or $-1$.  Then the number of sums of the form $\sum_{i=1}^k x_i \epsilon_i$ which fall into an arbitrary 
open
interval $I$ of length 2 does not exceed $\binom{k}{\lfloor k/2\rfloor}$.
\end{prop}
Taking $I=(-1,1)$, an immediate consequence is that for any integer partition $\lambda$ with $k$ parts, 
\begin{equation}\label{LEObound}
2^k r_\lambda = |\lambda^\perb| \leq \binom{k}{\lfloor \frac{k}{2}\rfloor},
\end{equation}
and in case $k=2m$ is even, the novel partition $\lambda=1^{2m}$ achieves equality with this upper bound.

A related theorem of Erd\H{o}s \cite{LEO} expands Proposition \ref{LEOthm} by widening the target interval. 

\begin{prop}[Erd\H{o}s \cite{LEO}]
\label{Erdosrbin}
Let $r$ be any integer, the $x_i$ real, $|x_i| \geq 1$. Then the number
of sums $\sum_{i=1}^k \epsilon_i x_i$ which fall into the interior of any interval of length
$2r$ is not greater than the sum of the r greatest binomial coefficients belonging
to $k$.
\end{prop}

This proposition was proved by showing that the size of the union of $r$ disjoint antichains in $\{-1,1\}^k$  is at most the sum of the $r$ largest binomial coefficients for $k$;  see \cite{TaoBook}, Proposition 7.7, and \cite{Bollobas}, Section 3, Exercise 7.   As a corollary of this, we get 

\begin{thm}
\label{nonequal bound}
Suppose $\lambda$ is an integer partition with $k$ parts, not all equal.  Then $2^k r_\lambda=|\lambda^\perb|$ is at most the sum of the largest four binomial coefficients of $k-2$.  Hence, for $k\ge 2$ and even, $\lambda=1^{k}$ has $|\lambda^\perb| >
|\mu^\perb|$ for any partition $\mu$ with $k$ parts, not all equal, and for $k \ge 5$ and odd, $\lambda=2\,1^{k-1}$ has $|\lambda^\perb| \ge
|\mu^\perb|$, for any partition $\mu$ with $k$ parts.
\end{thm}

\begin{proof}
Fix $i,j$ such that $\lambda_i \ne \lambda_j$.
Partition the set  $\lambda^\perb$ into four (possibly empty) subsets
\begin{eqnarray*}
A & = & \{x \in \lambda^\perb:  x_i =1, x_j = 1\}, \\
B & = & \{x \in \lambda^\perb: x_i =-1, x_j = -1\}, \\
C & = & \{x \in \lambda^\perb: x_i =1, x_j = -1\}, \\
D & = & \{x \in \lambda^\perb: x_i =-1, x_j = 1\}.
\end{eqnarray*}
These are disjoint antichains, since they specify four distinct target values for the sums $\sum' x_\ell \lambda_\ell$, where the sum is over the $k-2$ indices other than $i,j$, and each $\lambda_\ell$ is strictly positive.  Projecting out the two coordinates indexed by $i$ and $j$, we get 4 disjoint antichains in $\{-1,+1\}^{k-2}$.
\end{proof}

\begin{conjecture}\label{3rd place conj}{\bf Runners-up in Erd\H{o}s-Littlewood-Offord.}\\
For all partitions $\lambda$ with exactly $k$ parts, with greatest common divisor 1,
if $k  \ge 4$ is even, the second  largest probability $r_\lambda$ is achieved, \emph{uniquely}, by $2^2 1^{k-2}$, while if
$k \ge 7$ is odd, the largest probability is achieved by $2 \,1^{k-1}$ (already proved, as part of Theorem
\ref{nonequal bound}), and the second largest  is achieved, \emph{uniquely}, by $2^3 1^{k-3}$.
\end{conjecture}
We note that for $k\ge 5$ odd, it is trivial to check that $2 \,  1^{k-1}$ strictly beats $2^3 1^{k-3}$, and with $k=5$, 
$|22211^\perb|=|32111^\perb|=6$.

 \begin{prop}\label{prop no 3}
There are no novel partitions of size three.
\end{prop}

\begin{proof}
By Theorem \ref{rank equiv}, any novel partition $\lambda$ with three parts must have rank $A^{(\lambda)} = 2$.  Since 111 is not a valid template, the parts in $\lambda$ are not all equal, and so by Theorem \ref{nonequal bound}, $|\lambda^\perb| \leq 2$, which means that $A^{(\lambda)}$ has at most 1 column, and hence rank at most 1.
\end{proof}

\begin{prop}\label{prop no 4}
The only novel partition of size 4 is 1111.
\end{prop}

\begin{proof}
By Theorem \ref{rank equiv}, any novel partition $\lambda$ with four parts must have rank $A^{(\lambda)} = 3$.
By Theorem \ref{nonequal bound}, any novel partition $\lambda$ with four parts, not all equal, has $|\lambda^\perb| \leq 4$, which means that $A^{(\lambda)}$ has at most 2 columns, and hence rank at most 2.  If all parts of the partition are equal, then the requirement $\gcd(\lambda)=1$  forces $\lambda=1111$. This is indeed novel, with 
$A^{(1111)}$ given in Example \ref{ex display}.
\end{proof}

\begin{prop}\label{prop no 5}
The only novel partition of size 5 is 21111.
\end{prop}
\begin{proof}
Without loss of generality, assume $\lambda = (a,b,c,d,e)$, where $a\geq b \geq c \geq d \geq e > 0$.  As described in Example
\ref{example coin}, in order to avoid implying 11, every pair of parts in $\lambda$, when added or subtracted, must be a signed combination of the others, e.g., 
\begin{eqnarray*}
a + b & = & \pm c \pm d \pm e \\
b+c & = & \pm a \pm d \pm e.
\end{eqnarray*}
Let us look at the first equation.  If any of the signs are negative, then monotonicity is necessarily broken.  Thus, any novel partition of length five must have $a+b = c+d+e$.  Similarly, we can look at $b+c = \pm a \pm d \pm e$, and by a monotonicity argument we conclude that the only viable form is $b+c = a \pm d \pm e$.  We will look at each of these four cases separately.
\begin{enumerate}
\item \begin{eqnarray*}
a + b & = & c +d +e \\
b + c & = & a+d+e
\end{eqnarray*} can be refined (by adding or subtracting one from the other) to $b = d+e$ and $a=c$.  By monotonicity this means that $a=b=c$, and hence our partition would be of the form $(d+e, d+e, d+e, d, e)$.  However, we must have a solution to $d-e = \pm(d+e) \pm (d+e)\pm(d+e)$, which would imply that $e=0$.
\item  \begin{eqnarray*}
a + b & = & c +d +e \\
b + c & = & a+d-e
\end{eqnarray*} can be refined similarly to $b=d$ and $a=c+e$, which yields partitions of the form $(c+e, c, c, c, e)$.  We must have a solution to $c+2e = \pm c\pm c\pm c$, which, to avoid implying $e \le 0$, implies $e=c$, and our template reduces to a multiple of 21111.
\item  \begin{eqnarray*}
a + b & = & c +d +e \\
b + c & = & a-d+e
\end{eqnarray*}  can be refined to $b=e$, $a=c+d$, hence a multiple of 21111.
\item  \begin{eqnarray*}
a + b & = & c +d +e \\
b + c & = & a-d-e
\end{eqnarray*} forces $b=0$.
\end{enumerate}
\end{proof}

\begin{table}\label{table r}
$$
\tiny
\begin{array}{ccccc}
i & \lambda(i)  & len(\lambda)& r_\lambda & 256\cdot r_\lambda\\
1& \nullleftbrace1\nc1\nullrbrace & 2 & \ofrac{1}{2} & 128. \\
2& \nullleftbrace1\nc1\nc1\nc1\nullrbrace & 4 & \ofrac{3}{8} & 96. \\
3& \nullleftbrace1\nc1\nc1\nc1\nc1\nc1\nullrbrace & 6 & \ofrac{5}{16} & 80. \\
4& \nullleftbrace1\nc1\nc1\nc1\nc1\nc1\nc1\nc1\nullrbrace & 8 & \ofrac{35}{128} & 70. \\
5& \nullleftbrace2\nc1\nc1\nc1\nc1\nullrbrace & 5 & \ofrac{1}{4} & 64. \\
6& \nullleftbrace1\nc1\nc1\nc1\nc1\nc1\nc1\nc1\nc1\nc1\nullrbrace & 10 & \ofrac{63}{256} & 63. \\
7& \nullleftbrace2\nc1\nc1\nc1\nc1\nc1\nc1\nullrbrace & 7 & \ofrac{15}{64} & 60. \\
8& \nullleftbrace1\nc1\nc1\nc1\nc1\nc1\nc1\nc1\nc1\nc1\nc1\nc1\nullrbrace & 12 & \ofrac{231}{1024} & 57.75 \\
9& \nullleftbrace2\nc2\nc1\nc1\nc1\nc1\nullrbrace & 6 & \ofrac{7}{32} & 56. \\
10& \nullleftbrace2\nc1\nc1\nc1\nc1\nc1\nc1\nc1\nc1\nullrbrace & 9 & \ofrac{7}{32} & 56. \\
11& \nullleftbrace1\nc1\nc1\nc1\nc1\nc1\nc1\nc1\nc1\nc1\nc1\nc1\nc1\nc1\nullrbrace & 14 & \ofrac{429}{2048} & 53.625 \\
12& \nullleftbrace2\nc1\nc1\nc1\nc1\nc1\nc1\nc1\nc1\nc1\nc1\nullrbrace & 11 & \ofrac{105}{512} & 52.5 \\
13& \nullleftbrace2\nc2\nc1\nc1\nc1\nc1\nc1\nc1\nullrbrace & 8 & \ofrac{13}{64} & 52. \\
14& \nullleftbrace1\nc1\nc1\nc1\nc1\nc1\nc1\nc1\nc1\nc1\nc1\nc1\nc1\nc1\nc1\nc1\nullrbrace & 16 & \ofrac{6435}{32768} & 50.2734 \\
15& \nullleftbrace2\nc1\nc1\nc1\nc1\nc1\nc1\nc1\nc1\nc1\nc1\nc1\nc1\nullrbrace & 13 & \ofrac{99}{512} & 49.5 \\
16& \nullleftbrace2\nc2\nc1\nc1\nc1\nc1\nc1\nc1\nc1\nc1\nullrbrace & 10 & \ofrac{49}{256} & 49. \\
17& \nullleftbrace2\nc2\nc2\nc1\nc1\nc1\nc1\nullrbrace & 7 & \ofrac{3}{16} & 48. \\
18& \nullleftbrace1\nc1\nc1\nc1\nc1\nc1\nc1\nc1\nc1\nc1\nc1\nc1\nc1\nc1\nc1\nc1\nc1\nc1\nullrbrace & 18 & \ofrac{12155}{65536} & 47.4805 \\
19 &\nullleftbrace2\nc1\nc1\nc1\nc1\nc1\nc1\nc1\nc1\nc1\nc1\nc1\nc1\nc1\nc1\nullrbrace & 15 & \ofrac{3003}{16384} & 46.9219 \\
20 &\nullleftbrace2\nc2\nc1\nc1\nc1\nc1\nc1\nc1\nc1\nc1\nc1\nc1\nullrbrace & 12 & \ofrac{93}{512} & 46.5 \\
21 &\nullleftbrace2\nc2\nc2\nc1\nc1\nc1\nc1\nc1\nc1\nullrbrace & 9 & \ofrac{23}{128} & 46. \\
22 &\nullleftbrace1\nc1\nc1\nc1\nc1\nc1\nc1\nc1\nc1\nc1\nc1\nc1\nc1\nc1\nc1\nc1\nc1\nc1\nc1\nc1\nullrbrace & 20 & \ofrac{46189}{262144} & 45.1064 \\
23 &\nullleftbrace2\nc1\nc1\nc1\nc1\nc1\nc1\nc1\nc1\nc1\nc1\nc1\nc1\nc1\nc1\nc1\nc1\nullrbrace & 17 & \ofrac{715}{4096} & 44.6875 \\
24 &\nullleftbrace2\nc2\nc1\nc1\nc1\nc1\nc1\nc1\nc1\nc1\nc1\nc1\nc1\nc1\nullrbrace & 14 & \ofrac{1419}{8192} & 44.3438 \\
25 &\nullleftbrace3\nc2\nc1\nc1\nc1\nc1\nc1\nullrbrace & 7 & \ofrac{11}{64} & 44. \\
26 &\nullleftbrace2\nc2\nc2\nc2\nc1\nc1\nc1\nc1\nullrbrace & 8 & \ofrac{11}{64} & 44. \\
27 &\nullleftbrace2\nc2\nc2\nc1\nc1\nc1\nc1\nc1\nc1\nc1\nc1\nullrbrace & 11 & \ofrac{11}{64} & 44. \\
28 & \nullleftbrace1\nc1\nc1\nc1\nc1\nc1\nc1\nc1\nc1\nc1\nc1\nc1\nc1\nc1\nc1\nc1\nc1\nc1\nc1\nc1\nc1\nc1\nullrbrace & 22 & \ofrac{88179}{524288} & 43.0562 \\
29 & \nullleftbrace2\nc1\nc1\nc1\nc1\nc1\nc1\nc1\nc1\nc1\nc1\nc1\nc1\nc1\nc1\nc1\nc1\nc1\nc1\nullrbrace & 19 & \ofrac{21879}{131072} & 42.7324 \\
30 & \nullleftbrace2\nc2\nc1\nc1\nc1\nc1\nc1\nc1\nc1\nc1\nc1\nc1\nc1\nc1\nc1\nc1\nullrbrace & 16 & \ofrac{2717}{16384} & 42.4531 \\
31 & \nullleftbrace2\nc2\nc2\nc1\nc1\nc1\nc1\nc1\nc1\nc1\nc1\nc1\nc1\nullrbrace & 13 & \ofrac{675}{4096} & 42.1875 \\
32 & \nullleftbrace3\nc1\nc1\nc1\nc1\nc1\nc1\nc1\nullrbrace & 8 & \ofrac{21}{128} & 42. \\
33 & \nullleftbrace3\nc3\nc1\nc1\nc1\nc1\nc1\nc1\nullrbrace & 8 & \ofrac{21}{128} & 42. \\
34 & \nullleftbrace3\nc2\nc1\nc1\nc1\nc1\nc1\nc1\nc1\nullrbrace & 9 & \ofrac{21}{128} & 42. \\
35 & \nullleftbrace3\nc1\nc1\nc1\nc1\nc1\nc1\nc1\nc1\nc1\nullrbrace & 10 & \ofrac{21}{128} & 42. \\
36 & \nullleftbrace2\nc2\nc2\nc2\nc1\nc1\nc1\nc1\nc1\nc1\nullrbrace & 10 & \ofrac{21}{128} & 42. \\
 37 & \nullleftbrace1\nc1\nc1\nc1\nc1\nc1\nc1\nc1\nc1\nc1\nc1\nc1\nc1\nc1\nc1\nc1\nc1\nc1\nc1\nc1\nc1\nc1\nc1\nc1\nullrbrace & 24 & \ofrac{676039}{4194304} & 41.2621 \\
38 & \nullleftbrace2\nc1\nc1\nc1\nc1\nc1\nc1\nc1\nc1\nc1\nc1\nc1\nc1\nc1\nc1\nc1\nc1\nc1\nc1\nc1\nc1\nullrbrace & 21 & \ofrac{20995}{131072} & 41.0059 \\
39 & \nullleftbrace2\nc2\nc1\nc1\nc1\nc1\nc1\nc1\nc1\nc1\nc1\nc1\nc1\nc1\nc1\nc1\nc1\nc1\nullrbrace & 18 & \ofrac{10439}{65536} & 40.7773 \\
40 & \nullleftbrace3\nc2\nc1\nc1\nc1\nc1\nc1\nc1\nc1\nc1\nc1\nullrbrace & 11 & \ofrac{81}{512} & 40.5 \\
41 & \nullleftbrace2\nc2\nc2\nc2\nc1\nc1\nc1\nc1\nc1\nc1\nc1\nc1\nullrbrace & 12 & \ofrac{323}{2048} & 40.375 \\
42 & \nullleftbrace3\nc1\nc1\nc1\nc1\nc1\nc1\nc1\nc1\nc1\nc1\nc1\nc1\nc1\nullrbrace & 14 & \ofrac{1287}{8192} & 40.2188 \\
43 & \nullleftbrace3\nc1\nc1\nc1\nc1\nc1\nullrbrace & 6 & \ofrac{5}{32} & 40. \\
44 & \nullleftbrace3\nc2\nc2\nc1\nc1\nc1\nullrbrace & 6 & \ofrac{5}{32} & 40. \\
45 & \nullleftbrace3\nc2\nc2\nc2\nc1\nc1\nc1\nullrbrace & 7 & \ofrac{5}{32} & 40. \\
46 & \nullleftbrace3\nc2\nc2\nc1\nc1\nc1\nc1\nc1\nullrbrace & 8 & \ofrac{5}{32} & 40. \\
47 &\nullleftbrace2\nc2\nc2\nc2\nc2\nc1\nc1\nc1\nc1\nullrbrace & 9 & \ofrac{5}{32} & 40. \\
 48 &\nullleftbrace1\nc1\nc1\nc1\nc1\nc1\nc1\nc1\nc1\nc1\nc1\nc1\nc1\nc1\nc1\nc1\nc1\nc1\nc1\nc1\nc1\nc1\nc1\nc1\nc1\nc1\nullrbrace & 26 & \ofrac{1300075}{8388608} & 39.6751 \\
 49 &\nullleftbrace2\nc1\nc1\nc1\nc1\nc1\nc1\nc1\nc1\nc1\nc1\nc1\nc1\nc1\nc1\nc1\nc1\nc1\nc1\nc1\nc1\nc1\nc1\nullrbrace & 23 & \ofrac{323323}{2097152} & 39.4681 \\
50 &\nullleftbrace2\nc2\nc1\nc1\nc1\nc1\nc1\nc1\nc1\nc1\nc1\nc1\nc1\nc1\nc1\nc1\nc1\nc1\nc1\nc1\nullrbrace & 20 & \ofrac{20111}{131072} & 39.2793 \\
51 &\nullleftbrace3\nc2\nc1\nc1\nc1\nc1\nc1\nc1\nc1\nc1\nc1\nc1\nc1\nullrbrace & 13 & \ofrac{627}{4096} & 39.1875 \\
52 &\nullleftbrace3\nc1\nc1\nc1\nc1\nc1\nc1\nc1\nc1\nc1\nc1\nc1\nc1\nc1\nc1\nc1\nullrbrace & 16 & \ofrac{5005}{32768} & 39.1016 \\
53 &\nullleftbrace3\nc2\nc2\nc1\nc1\nc1\nc1\nc1\nc1\nc1\nullrbrace & 10 & \ofrac{39}{256} & 39. \\
54 &\nullleftbrace3\nc3\nc1\nc1\nc1\nc1\nc1\nc1\nc1\nc1\nullrbrace & 10 & \ofrac{39}{256} & 39. \\
55 &\nullleftbrace2\nc2\nc2\nc2\nc1\nc1\nc1\nc1\nc1\nc1\nc1\nc1\nc1\nc1\nullrbrace & 14 & \ofrac{623}{4096} & 38.9375 \\
56 &\nullleftbrace2\nc2\nc2\nc2\nc2\nc1\nc1\nc1\nc1\nc1\nc1\nullrbrace & 11 & \ofrac{155}{1024} & 38.75 \\
 57 &\nullleftbrace1\nc1\nc1\nc1\nc1\nc1\nc1\nc1\nc1\nc1\nc1\nc1\nc1\nc1\nc1\nc1\nc1\nc1\nc1\nc1\nc1\nc1\nc1\nc1\nc1\nc1\nc1\nc1\nullrbrace & 28 & \ofrac{5014575}{33554432} & 38.2582 \\
 58 &\nullleftbrace2\nc1\nc1\nc1\nc1\nc1\nc1\nc1\nc1\nc1\nc1\nc1\nc1\nc1\nc1\nc1\nc1\nc1\nc1\nc1\nc1\nc1\nc1\nc1\nc1\nullrbrace & 25 & \ofrac{156009}{1048576} & 38.0881 \\
59 & \nullleftbrace3\nc2\nc2\nc2\nc1\nc1\nc1\nc1\nc1\nullrbrace & 9 & \ofrac{19}{128} & 38. \\


\end{array}
$$
\caption{ 
Novel partitions sorted by $r_\lambda$.  Conjectured to be complete with respect to $r_\lambda \ge 38/256$.  }
\end{table}

\begin{prop}\label{prop no 7}
The only novel partitions of length 6 are 111111, 221111, 311111, 322111.
The only novel partitions of length 7 are

$$
\begin{array}{ccccccc}
 2 & 1 & 1 & 1 & 1 & 1 & 1 \\
 2 & 2 & 2 & 1 & 1 & 1 & 1 \\
 3 & 2 & 1 & 1 & 1 & 1 & 1 \\
 3 & 2 & 2 & 2 & 1 & 1 & 1 \\
 3 & 3 & 2 & 1 & 1 & 1 & 1 \\
 3 & 3 & 2 & 2 & 2 & 1 & 1 \\
 4 & 1 & 1 & 1 & 1 & 1 & 1 \\
 4 & 2 & 2 & 1 & 1 & 1 & 1 \\
 4 & 3 & 2 & 2 & 1 & 1 & 1 \\
 4 & 3 & 3 & 1 & 1 & 1 & 1 \\
 4 & 3 & 3 & 2 & 2 & 1 & 1 \\
 5 & 2 & 2 & 2 & 1 & 1 & 1 \\
 5 & 3 & 3 & 2 & 1 & 1 & 1 \\
 5 & 4 & 3 & 2 & 2 & 1 & 1
\end{array}
$$

\end{prop}

\begin{proof}
The same technique that was used in Proposition \ref{prop no 5} can be continued for novel partitions of length 6, 7, etc., eliminating cases that imply 11.  {\tt Mathematica} \cite{mathematica} code was written to list all cases and reduce them. For the sake of economy in running time, we only considered the requirement that
all four of $\lambda_1 \pm \lambda_2$ and $\lambda_{k-1} \pm \lambda_k$ be expressible plus-minus combination of the other $k-2$ parts.   When the reduction yields a space of dimension greater than one, the result may be viewed as what we call a meta-template, e.g., $(a+b, a+b, b, b, a, a)$.  The list of meta-templates includes all novel partitions and possibly others that are not novel.
For $k=6$, the following candidates were returned: $$111111, 221111, 311111, 322111, 332211, 433211, 533221.$$
We showed in Example \ref{example strict} that $332211\strict 221111$, and the ranks of 
$A^{(433211)}$ and $A^{(533221)}$
are 4, whereas the others have rank 5.

For $k=7$, a list of 14 templates was found; all of which turned out to be novel.  Also, for $k=7$, 12 meta-templates were found;
by hand inspection, 11 were easily shown to  violate the monotonicity requirement that
$\lambda_1 \ge \lambda_2 \ge \cdots$.  The remaining meta-template, $(a,b,a-b,d,e,h,h-d-e)$ is seen, by hand, to
imply 11,  since after the initial refinement to the form above one can apply the same technique \emph{again} to the smallest two parts and reduce each case to either a monotonicity or positivity violation.
\end{proof}
 
\begin{table}
$$
\begin{array}{cc|cc|cc}
\lambda & |\lambda^\perb| & \lambda & |\lambda^\perb| & \lambda & |\lambda^\perb|\\
\hline
11111111&70&54433221&22&65522211&16 \\
 22111111&52&53332211&22&65533211&16 \\
 22221111&44&65322211&22&65544332&16 \\
 33111111&42&64322111&22&76433221&16 \\
 31111111&42&64432111&22&76533211&16 \\
 32211111&40&63222111&22&76543221&16 \\
 33221111&36&55422211&20&76544321&16 \\
 32222111&36&55433211&20&75543211&16 \\
 33222211&34&65332111&20&87433221&16 \\
 43211111&32&65432211&20&86433211&16 \\
 42111111&32&64331111&20&86543211&16 \\
 42221111&32&64332211&20&85432211&16 \\
 33311111&30&63321111&20&85542211&16 \\
 33322111&30&63322211&20&84332211&16 \\
 44221111&30&44333111&18&55443331&14 \\
 43222111&30&55333111&18&54411111&14 \\
 43321111&30&55443221&18&53332222&14 \\
 43322211&28&54433111&18&51111111&14 \\
 33322221&26&54433322&18&76433111&14 \\
 44322111&26&65332221&18&76522211&14 \\
 44332211&26&65422111&18&76544211&14 \\
 43332111&26&65433111&18&76554331&14 \\
 43332221&26&65433221&18&75443322&14 \\
 54222111&26&65443211&18&75522111&14 \\
 53221111&26&65443321&18&74333222&14 \\
 53322111&26&65543221&18&74431111&14 \\
 53331111&26&64421111&18&73331111&14 \\
 52222111&26&64433211&18&72222111&14 \\
 54321111&24&62221111&18&87533211&14 \\
 54322211&24&76332221&18&87543221&14 \\
 54332111&24&76432211&18&87654321&14 \\
 53222211&24&75332211&18&86533111&14 \\
 53311111&24&75432111&18&85532111&14 \\
 52211111&24&75433211&18&83332111&14 \\
 44311111&22&75442211&18&98543221&14 \\
 44333221&22&75533111&18&97543211&14 \\
 55322111&22&74322211&18&97644211&14 \\
 54332221&22&74422111&18&96542211&14 \\
 54333211&22&74432211&18&95532211&14 \\
 54422111&22&73322111&18&94432211&14 \\
 54432211&22&73332211&18
\end{array}
$$
\caption{Novel partitions of length 8, conjectured to be the complete list.}
\label{table list 8}
\end{table}

\begin{lemma}
\label{l1}
In order of decreasing $r_\lambda$, the first eight novel  partitions $\lambda$ are $11, 1111, 1^6, 1^8, 21111, 1^{10}, 2\,1^6, 1^{12}$.  For novel partitions other than these eight, 
writing $k$ for the number of parts
\begin{eqnarray*}
k=6,7: & & r_\lambda \leq \frac{60}{256}, \\
 k=8,9: & & r_\lambda \leq \frac{56}{256}, \\
k=10,11: & & r_\lambda \leq \frac{52.5}{256},\\
\mbox{not }1^{14}, 1^{16}, \mbox{ and }  k\geq 12: & & r_\lambda \leq \frac{49.5}{256}.
\end{eqnarray*}

Hence, aside from the first  eight novel partitions, all other novel partitions have $r_\lambda \leq 56/256$.
Observe that $\lambda=1^{14}$ has $r_\lambda=53.625/256$ and $\lambda=1^{16}$ has $r_\lambda=50.2734375/256$.
\end{lemma}

\begin{proof}
This follows immediately from Theorem \ref{nonequal bound}.  For example, when $k=6$ the four largest binomial coefficients of $k-2$ appear on the left side below, and
$$
\binom{4}{2}+\binom{4}{1}+\binom{4}{3}+\binom{4}{0} = 15,
$$
with $15/2^6 = 60/256$ giving our upper bound for $k=6$.  

\end{proof}
   
\begin{conjecture}\label{conj r}   
In order of decreasing $r_\lambda$, the  novel  partitions
with $r_\lambda \ge 38/256$ are precisely those given in Table \ref{table r}.
\end{conjecture}
   
\begin{example}\label{8 progression}{\bf The shortest novel arithmetic progression.}\\
The partition $\lambda = (8,7,6,5,4,3,2,1)$ is novel.  It has $|\lambda^\perb|=14$, so $A^{(\lambda)}$ is an 8 by 7 matrix of rank 7.   Examination of the $21= 1+0+1+1+4+14$ novel partitions of lengths 2,3,4,5,6,7  in Propositions
\ref{prop no 4} -- \ref{prop no 7} shows that this $\lambda$ is the shortest
novel partition which is also an arithmetic progression.
\end{example}

\begin{conjecture}\label{8 conj}
There are exactly 
122 novel partitions of length $k=8$. 
\end{conjecture}
\noindent The list of 122 is given in Table \ref{table list 8}.
Our evidence in favor of this conjecture is that these 122, and no others, were found by a random survey, using {\tt Mathematica}, of 420 million singular $n$ by $n$ matrices $M$, for $n=8$.  Of course, this is not a proof.   For an exhaustive search,  to guarantee that all novel partitions of length 8 have been found, one might observe that, with respect to the integer partitions underlying potential right and left null vectors, $M$ can be taken
to have first row and first column all $+1$, so that it would suffice to examine $2^{49}$ matrices $M$.

\begin{rem}
The {\tt Mathematica} command {\tt NullSpace} applied to a singular $n$ by $n$ Bernoulli matrix $M$ returns a list of length $n$ vectors that forms \emph{a} basis for the null space of $M$.  Aside from the sign requirement in the first nonzero entry, these vectors have always been of the form $v\in \Vn$ for some \emph{novel} partition $\lambda$. One would like to prove a result about this, but since the basis returned by a generic null space algorithm is not unique, and hence  implementation
dependent, we will not pursue this idea further.
\end{rem}

\section{Polynomial coefficients arising from inclusion-exclusion}\label{sect ie}

For events  $\{A_\alpha\}_{\alpha\in I}$, for a finite index set $I$, and $A=\cup_{\alpha\in I} A_\alpha$,
the inclusion-exclusion formula states that
\begin{equation}
\label{incexc}\p(A)  =  \sum_{\alpha \in I} \p(A_\alpha) -\sum_{\{\alpha,\beta\}\subset I,\alpha\neq\beta} \p(A_\alpha\cap A_\beta) + \sum \p(A_\alpha\cap A_\beta \cap A_\gamma)
\end{equation}
$$
+ \cdots + (-1)^{|I|-1} \p(\cap_{\alpha \in I} A_\alpha).
$$
With $W = \sum_{\alpha \in I} 1(A_\alpha)$, a sum of indicators of the events, the formula above may be expressed as
\begin{equation}
\label{incexcE}
\p(A)  =   \e W - \e {W \choose 2} +  \e {W \choose 3} + \cdots + (-1)^{|I|-1} \e {W \choose |I|}.
\end{equation}

The Bonferroni inequalities state that for events  $\{A_\alpha\}_{\alpha\in I}$, for a finite index set $I$, and $A=\cup_{\alpha\in I} A_\alpha$,

\begin{eqnarray}
\nonumber \p(A) & \leq & \sum_{\alpha\in I} \p(A_\alpha) \\
\label{lowerbound}\p(A) & \geq & \sum_{\alpha\in I} \p(A_\alpha) - \sum_{\{\alpha,\beta\}\subset I,\alpha\neq\beta} \p(A_\alpha\cap A_\beta) \\
\nonumber \cdots & \cdots & \cdots .
\end{eqnarray}

Equation \ref{lowerbound} is a lowerbound, with the $\ldots$ representing higher order bounds.  
A variation of \eqref{lowerbound}, with $B=\cup_{\beta\in I'} A_\beta$, 
\begin{equation}\label{lowerbound set diff}
\p(A\setminus B) \ge \sum_{\alpha\in I} \p(A_\alpha)  - \sum_{\{\alpha,\beta\}\subset I,\alpha\neq\beta} \p(A_\alpha\cap A_\beta) - \sum_{\alpha \in I, \beta \in I'} \p(A_\alpha \cap A_\beta), \\
\end{equation} 
is proved similarly. 


We take $I = {[n] \choose 2} \times \{-,+\} \times \{L,R\}$,\footnote{The notation used here is:  $[n]$ is the set $\{1,2,\ldots,n\}$, and for a set $T$ and nonnegative integer $k$, ${T \choose k}$ is the set of all $k$-subsets of $T$.} so that $\alpha \in I$ specifies a set of two distinct indices along with sign and direction bits. The event $A_\alpha$ corresponds to the occurrence of a null vector of the form $\alpha$.  For example, $\alpha = ( \{2,5\},-,R)\in I$, and $A_\alpha$ is the event that $e_2-e_5$ is a right null vector.  
\begin{prop}
For $W = \sum_{\alpha \in I} 1(A_\alpha)$ with $I$ and the $A_\alpha$ as above, so that $D_{11}=\{W > 0 \}$,
\label{propie}
\begin{eqnarray*}
\e  W & = & 4\binom{n}{2} \left(\frac{1}{2}\right)^n,\\
\e  \binom{W}{2} & = & \left(12\binom{n}{2}^2 - 4
\binom{n}{2}\right) \left(\frac{1}{4}\right)^n,\\
\e  \binom{W}{3} & = & 2^2\binom{n}{3} \left(\frac{1}{4}\right)^{n}+ \\
 & & + \, 2^{3-3n}\left(\frac{13}{3}\, \binom{n}{2}^3 -4\binom{n}{2}^2  - \frac{2}{3}\,\binom{n}{2}  - \frac{1}{3}\,\binom{n}{3}\binom{3}{2}^3 - \binom{n}{2}\binom{n-2}{2}\right)\\
 & = & 4 \binom{n}{3}\left(\frac{1}{4}\right)^n+O(n^6 2^{-3n}).
\end{eqnarray*}
\end{prop}

\begin{proof}
Let $t = 2\binom{n}{2} 2^{-n}$.  Clearly,
\begin{equation}
\nonumber\e  W = \sum_{\alpha\in I} P(e_i = \pm e_j)= |I| 2^{-n} = 4\binom{n}{2} 2^{-n} = 2t.
\end{equation}

Let $I_R$ (resp. $I_L$) be the set of $\alpha \in I$ with last coordinate \mbox{$R$ (resp. $L$).}  Let $B_\alpha = 1(A_\alpha)$ be the indicator random variable, for $\alpha \in I$.  Then
\begin{eqnarray}
\nonumber\e  \binom{W}{2} & = & 2 \sum_{\{\alpha,\beta\} \subset I_R,\alpha \ne \beta} \e  B_{\alpha}B_\beta + \sum_{\alpha\in I_R,\beta\in I_L} \e  B_\alpha B_\beta =: G_1 + G_2.
\end{eqnarray}
The first sum corresponds to  two null vectors both on the right; with a factor of 2 we get the contribution $G_1$ for both null vectors on the same side, either right or left.  The second contribution $G_2$ corresponds to two null vectors on opposite sides.  We have
\begin{eqnarray}
\label{eqG1}G_1 & = & t^2 - 2^2\binom{n}{2} 2^{-2n} = t^2 - t 2^{1-n},\\
\label{eqG2}G_2& = & 2\binom{n}{2} 2^{-n} \times 2\binom{n}{2}2^{-n} \times 2 = 2t^2.
\end{eqnarray}

The first equation \eqref{eqG1} is found by considering all pairs on one side, and then taking away all pairs that share two rows along with any plus or minus combination.  The second equation \eqref{eqG2} 
has a ``boost'' factor of 2, since for $\alpha,\beta$ on opposite sides, both of the template 11, we have
$\p(A_\alpha | A_\beta) = 2 \p(A_\alpha)$.

Finally, we have
\begin{equation}
\nonumber
6\, \e \binom{W}{3} = \sum_{(\alpha,\beta,\gamma)} 1(\mbox{$\alpha,\beta,\gamma$ distinct})\ \e B_\alpha B_\beta B_\gamma =: F_1 + F_2
\end{equation}
Here $F_1$ denotes the sum over all events where $\alpha, \beta, \gamma$ appear on the same side, and $F_2$ denotes the sum over events where two appear on one side, and one on the other.  By choosing a side (left or right), we have for some $t_2$, $t_3$, $t_4$ functions of $n$,
\begin{equation}
\nonumber
\frac{F_1}{2} = t^3 - t_2 - t_3 - t_4+2^2\binom{3}{2}\binom{n}{3} \left(\frac{1}{2}\right)^{2n}.
\end{equation}
The $t^3$ considers all triplets, and the $t_i$ considers the triplets that are supported on $i$ rows, $i=2,3,4$, that need to be excepted.  We have
\begin{equation}
\nonumber
t_2 = \binom{n}{2} 2^3 2^{-3n},
\end{equation}
which chooses any two rows and all sign combinations.  When three rows are supported, there are precisely $2^3 \binom{3}{1}^3$ combinations, but events of the form \{$e_i \pm e_j$, $e_i \pm e_k$, $e_j \pm e_k$ are null vectors\} are sometimes valid.  When they are valid, they have a probability of $(1/2)^{2n}$, hence excepting \emph{all} events involving three rows we have
\begin{equation}
\nonumber
t_3 = \binom{n}{3} 2^3 \binom{3}{1}^3 2^{-3n},
\end{equation} 
and the term at the end adds back in the valid combinations supporting three rows.  These events are of the form \{$e_i \pm e_j$ and $e_i \pm e_k$ are null vectors\}, which imply one of $e_j + e_k$ or $e_j - e_k$ is a null vector as well.
Finally, when four rows are supported, the exceptional cases are those in which two of $\alpha, \beta, \gamma$ share two rows, and one does not share any, thus
\begin{equation}
\nonumber
t_4 = \binom{3}{1} \binom{n}{2}\binom{n-2}{2} 2^3 2^{-3n}.
\end{equation}

For $F_2$, there are two choices for which side the solo index appears, and then three choices for which of $\alpha, \beta, \gamma$ is this solo index.  We have
\begin{equation}
\nonumber
\frac{F_2}{6} = 4t^3 - \binom{n}{2} 2^2 \times \binom{n}{2} 2\times 2^{2-3n}.
\end{equation}
The factor of four comes from $\e B_\alpha B_\beta B_\gamma = 4 \e B_\alpha\e B_\beta\e B_\gamma$, which is a boost from conditioning on an opposite side.  The exceptional cases are those where the support of the non-solo pair lie on the same two rows, and includes all sign combinations.  The solo index can be anything, and gets a conditional boost from being on the other side.
\end{proof}
 
\begin{prop}
\label{propR11minusL11}
Recall \eqref{def D}, and that $R_\lambda$ (resp. $L_\lambda$) denotes the event that there is a right (resp. left) null vector of template $\lambda$.  We have

\begin{eqnarray}
\nonumber \p(R_{11} \setminus L_{11}) & \geq & 2 \binom{n}{2} \left(\frac{1}{2}\right)^n - \left(12\binom{n}{2}^2- 4\binom{n}{2} \right) \left(\frac{1}{4}\right)^{n} \\
\nonumber \p(R_{11} \setminus L_{11}) &=& P(R_{11}) -  8\binom{n}{2}^2 \left(\frac{1}{4}\right)^{n} + O(n^6 2^{-3n}).
\end{eqnarray}
\end{prop}

\begin{proof}
The expansion follows along the same reasoning as Proposition~\ref{propie} using \eqref{lowerbound set diff}; in particular, with $t=2\binom{n}{2} 2^{-n}$ as before, we have
\begin{equation}
\nonumber\p(R_{11}\setminus L_{11}) \geq t - G_1 - G_2.
\end{equation}
\end{proof}

A similar analysis can be undertaken for $D_{1111}$; we omit the details.
 \begin{lemma}
\label{l2}
\begin{eqnarray*}
P(D_{11}) & \geq & 4\binom{n}{2} \left(\frac{1}{2}\right)^n - \left(12\binom{n}{2}^2 - 4\binom{n}{2}\right) \left(\frac{1}{4}\right)^n. \\
P(D_{11}) & = & 4\binom{n}{2} \left(\frac{1}{2}\right)^n - \left(12\binom{n}{2}^2 - 4\binom{n}{2}-4\binom{n}{3}\right) \left(\frac{1}{4}\right)^n +  O(n^6 2^{-3n}),
\end{eqnarray*}
\begin{eqnarray*}
\p(D_{1111}) & = & 2^4\binom{n}{4}(3/8)^n +O(n^5(3/16)^n), \\
\p(D_{1^6}) & = & 2^6\binom{n}{6}\left( \frac{5}{16} \right)^n + O\left( n^7\left( \frac{5}{32} \right)^n \right),
\end{eqnarray*}
\begin{eqnarray*}
\p(D_{1^8}) & = & 2^8 \binom{n}{8} \left( \frac{35}{128} \right)^n + O\left( n^9\left( \frac{35}{256} \right)^n \right),\\
\p(D_{21111}) & = & 2^5 \binom{n}{5}\binom{5}{1} \left(\frac{1}{4}\right)^n + O\left(n^6\left(\frac{1}{8}\right)^n\right), \\
\p(D_{1^{10}}) & = & 2^{10} \binom{n}{10} \left( \frac{63}{256} \right)^n + O\left( n^{11}\left( \frac{63}{512} \right)^n \right),\\
\p(D_{2\,1^6}) & = & 2^7 \binom{n}{7}\binom{7}{1} \left(\frac{60}{256}\right)^n + O\left(n^8\left(\frac{60}{512}\right)^n\right),
\end{eqnarray*}
\begin{eqnarray*}
\p(D_{1^{12}}) & = & 2^{12}\binom{n}{12}\left( \frac{231}{1024} \right)^n + O\left( n^{13}\left( \frac{231}{2048} \right)^n \right).\\
\end{eqnarray*}
\end{lemma}

\begin{proof}
The first two equations follow from Proposition \ref{propie}.  The rest are proved similarly, but we omit the details.
\end{proof}

Next we move to probabilities $\p(D_{\lambda} \cap D_{\mu})$ for various choices of $\lambda \ne \mu$.  
Observe that the events involved can be highly positively correlated;  for example, with $\lambda=11,\mu=1111$ we have
$\p(D_{\lambda} \cap D_{\mu})/( \p(D_{\lambda}) \times \p(  D_{\mu}))$ grows exponentially fast, as $(4/3)^n$.

\begin{prop}
\label{proplambda}
For two distinct novel partitions $\lambda, \mu$, having $j$ and $k$ parts, respectively,

$$\p(D_\lambda\cap D_{\mu}) \leq O\left(n^{k+j}(\max(r_\lambda, r_\mu)/2)^n\right).$$
The implicit constant in the big O varies with the choice of $\lambda,\mu$.
\end{prop}

\begin{proof}
Consider the event $R_\lambda \cap R_\mu= \bigcup \{ vM = wM=0 \}$,  where the union is over $v \in \Vn, w \in V_\mu^{(n)}$.
The crucial ingredient is to show, with the notation of \eqref{def r}, that
\begin{equation}\label{max}
   \p( v \cdot X=w \cdot X=0) \le \max(r_\lambda, r_\mu)/2.
\end{equation}
Without loss of generality, assume that the nonzero components of $v$ are indexed by $J$, so $|J|=j$, and the nonzero
components of $w$ are indexed by $K$, with $|K|=k$.  With $I=J \cup K$ having size $m = |I|$, the event of interest is based on $m$ independent fair
coins $\epsilon_i,  \ i \in I$, and can be expressed as 
\begin{equation}\label{event}
\left\{\sum_{i \in J} v_i \epsilon_i = 0\right\} \cap \left\{\sum_{i \in K} w_i \epsilon_i = 0\right\}.
\end{equation}

Case 1:  $J \ne K$.   Without loss of generality, interchanging the $\lambda$ and $\mu$ if needed, 
$I=\{1,2,\ldots,m\}$ and $ m \in K \setminus J$.  Condition on the values of the first $m-1$ coins, with a configuration that satisfies $\sum_{j \in J} v_j =0$.  These configurations belong to the event $vM=0$, and hence have probabilities summing to at most $r_\lambda$.  Each configuration, together with the requirement $wM=0$, dictates the value needed for $\epsilon_m$, which occurs with conditional probability $1/2$.  The possible exchange of $\lambda,\mu$ at the start  means that we have shown \eqref{max}.

Case 2: $J=K$.  Without loss of generality, rearranging the coordinates, and taking scalar multiples if needed, we can have $J=K=\{1,2,\ldots, k\}$ and  $a :=v_k=w_k \ne 0$.  The event in \eqref{event} simplifies to
$$
   \left\{ -a \epsilon_k = \sum_{i=1}^{k-1} v_i \epsilon_i = \sum_{i=1}^{k-1} w_i \epsilon_i  \right\}.
$$
From this we conclude
\begin{eqnarray*}
r_\lambda &=& \p(v\cdot X = 0) = \p\left(\sum_{i=1}^{k-1} v_i \epsilon_i \in \{\pm a\}\right) \\
 &\ge & 2 \ \p\left(\sum_{i=1}^{k-1}v_i \epsilon_i = \sum_{i=1}^{k-1} w_i \epsilon_i \in \{\pm a\}\right);
\end{eqnarray*}
inequality arises since the second sum, with weights $w_i$,  might not even be in $\{\pm a\}$, and the factor of 2 arises since
when the second sum is in the set, it dictates the choice of sign. 

For the case where the potential null vectors are used on opposite sides, e.g., $L_\lambda \cap R_{\mu}$, 
we have 
$$
\p(L_\lambda \cap R_{\mu}) = O( \ \p(L_\lambda) \times \p(R_{\mu}) \ ),
$$
for the simple reason that conditioning of events of the form $Mw=0$ with $w \in V_\mu^{(n)}$ only affects $k$ of the columns,
giving the bound above, with constant $(1/r_\mu)^k$ as the implicit constant for the big O.
\end{proof}

Proposition \ref{proplambda} involves a bound that can have exponential decay as large as $(1/4)^n$.
For the sake of proving Theorem \ref{thm1}, with error term involving $(7/32)^n$, we need a stronger bound, as given below.
\begin{lemma}
\label{l3}
For all novel partitions $\lambda, \mu$, having $j$ and $k$ parts, respectively, with $\lambda \ne \mu$, and neither partition equal to the partition $11$, we have,
\begin{eqnarray}
\label{D11andD1111}\p(D_{11}\cap D_{1111}) & = & 2^3\binom{n}{4}\left(\frac{1}{4}\right)^n +O\left(n^5\left(\frac{3}{16}\right)^n\right), \\
\label{D11andlambda}\p(D_{11}\cap D_{\lambda}) & = & O\left( n^{j+2}\left(\frac{3}{16}\right)^n \right), \\
\label{Dlambdalambdaprime}\p(D_{\lambda}\cap D_{\mu}) & = & O\left( n^{k+j}\left(\frac{3}{16}\right)^n \right).  
\end{eqnarray}
\end{lemma}

\begin{proof}
Equation \eqref{D11andD1111} can be computed directly using inclusion exclusion, whereas Equations \eqref{D11andlambda} and \eqref{Dlambdalambdaprime} use Proposition \ref{proplambda} with $\lambda=1111$ since it is the most likely partition after 11.
\end{proof}

Finally, we note a trivial lemma to simplify the coefficient for $4^{-n}$ in the expansion of Theorem \ref{thm1},

\begin{lemma}
\label{lemmatrivial}
\begin{equation}
\frac{1}{2}\left(\binom{n}{2}^2 - \binom{n}{2}\right) = 3\binom{n}{4}+3\binom{n}{3}.
\end{equation}
\end{lemma}

\begin{proof}
Either simplify algebraically or note that the left hand side is the number of ways to choose any two unordered distinct pairs of unordered distinct pairs of $n$ objects.  The right hand side counts the number of ways to select these pairs where all four indices are distinct and can be placed in 3 distinct configurations, and the second term counts the number of pairs that share a common index, of which there are 3 choices for the repeated index.
\end{proof}

\section{Interaction of left and right null vectors}\label{sect left right}

Proposition \ref{propR11minusL11} gives a lower bound on $\p(R_{11} \setminus L_{11}) = \p(L_{11} \setminus R_{11})$
 which has, as a corollary, 
$$
\p(S \setminus L_{11}) \ge \p(R_{11} \setminus L_{11}) \sim \p(R_{11}) =  \p(L_{11}),
$$
and omitting the middle terms, and writing  $a_n \gtrsim c_n $ to mean that there exists $b_n$ with  $a_n \ge b_n$ and $b_n \sim c_n$, we have
\begin{equation}\label{gtrsim}
\p(S \setminus L_{11}) \gtrsim  \p(L_{11}).
\end{equation}
Expressing \eqref{gtrsim} in terms of left null vectors, with the outer union on the left taken over all novel partitions $\lambda$ of length less than or equal to $n$, other than 11, we have
$$
\p\left( \bigcup_{\lambda\ne 11} \bigcup_{v \in \Vn} \{ vM=0\} \right) \gtrsim 
\p\left(  \bigcup_{v \in V_{11}^{(n)}} \{ vM=0\} \right). 
$$
Writing this with $L_\lambda$ for the event that $M$ has a left null vector of template $\lambda$,
the above display can be rewritten as
$$
\sum_{\lambda\ne 11} \p(L_\lambda) \gtrsim \p(L_{11}).
$$

We believe that  to prove sharp upper bounds on $P_n$, say as given by \eqref{Komlos1} or \eqref{conj 1 line 1},
it will be necessary to  consider the effect of conditioning on $D_{11}^c$.  Propositions \ref{prop no R11} and \ref{prop no 1111} might be a  first step in this direction.

\begin{prop}\label{prop no R11}
Suppose that $\lambda$ is a novel partition of length $k$, with $k=n$.  Let $2p = |\lambda^\perb |$. 
Recall that $R_{11}$ is the event that our $n$ by $n$ matrix $M$ has a right null vector of the form $e_i \pm e_j$.
For every $v \in V_\lambda$,
\begin{equation}\label{falling bound}
\nonumber\frac{\p(vM=0 | R_{11}^c)}{\p(vM=0)} = \frac{(p)_n}{p^n}.
\end{equation}
\end{prop}
\begin{proof}
The hypothesis $k=n$ is essential:  if $x$ denotes a column of $M$, then, thanks to $k=n$, we know that $x \in v^\perb$.
There are $p$ choices for the ``direction'' $\{-x,x\}$ with $x \in v^\perb$, and different columns of $M$ must choose different
directions, otherwise the event $R_{11}$ would occur.  By giving the ratio of the conditional probability to the unconditional probability, factors of 2, for choosing between $x$ and $-x$, for each column, cancel. 
\end{proof}

\begin{prop}\label{prop no 1111}
Suppose that $\lambda$ is a novel partition of length $k$, with $k=n-1$.  Let $2p = |\lambda^\perb |$. 
Recall that $R_{1111}$ is the event that our $n$ by $n$ matrix $M$ has a right null vector of the form 
$e_{j_1} \pm e_{j_2} \pm e_{j_3} \pm e_{j_4}$.
For every $v \in V_\lambda^{(n)}$, as specified by Definition \ref{def n template},
\begin{equation}\label{falling bound2}
\nonumber \frac{\p(vM=0 | (R_{11}\cup R_{1111})^c)}{\p(vM=0)} =\frac{(p)_n}{p^n}+ \frac{1}{2p}{n \choose 2} \frac{(p)_{n-1}}{p^{n-1}}.
\end{equation}
\end{prop}
\begin{proof}
The hypothesis $k=n-1$ is essential.  Without loss of generality, assume that $v_n=0$, so $v=(w,0)$ with $w \in V_\lambda$.  Let $x = (y,s)$ denote a column of $M$, where $y$ gives the first $n-1$ coordinates, and $s \in \{-1,+1\}$. Then, thanks to $k=n-1$ and $v_n=0$, we know that $y \in w^\perb$.
There are $p$ choices for the ``direction'' $\{-y,y\}$ --- restricting to the first $n-1$ coordinates, with $y \in w^\perb$, and different columns of $M$ must choose different directions, apart from possibly one pair of columns, where the columns in a pair may share the underlying $n-1$ direction, but have opposite choices of $s$ for their $n$th coordinate.
(If three columns share the underlying $n-1$ direction, the event $R_{11}$ would occur;  if two pairs of columns share, then event $R_{1111}$ would occur.)
\end{proof}

\section{Acknowledgement}

We would like to thank Philip Matchett Wood for giving a very inspirational talk that motivated the current document.  We are grateful to the anonymous referees for their careful readings and suggestions for improvement.

\bibliographystyle{plain}	
\bibliography{IPBibliography}		

\end{document}